\documentclass[a4paper,UKenglish,cleveref,autoref,thm-restate]{socg-lipics-v2021}

\usepackage{tikz,graphicx}
\usepackage{subcaption}
\usetikzlibrary{positioning,arrows.meta}

\usepackage{lineno}

\linenumbers

\title{On Threshold Pairwise Compatibility Graphs}

\author{Sheikh Azizul Hakim}
{Department of Computer Science and Engineering, Bangladesh University of Engineering and Technology, Dhaka, Bangladesh}
{hakim@cse.buet.ac.bd}
{https://orcid.org/0000-0002-3405-2402}
{}

\author{Md. Shamsuzzoha Bayzid}
{Department of Computer Science and Engineering, Bangladesh University of Engineering and Technology, Dhaka, Bangladesh}
{shams_bayzid@cse.buet.ac.bd}
{https://orcid.org/0000-0002-5640-0615}
{}

\authorrunning{S.~A.~Hakim and M.~S.~Bayzid}
\Copyright{Sheikh Azizul Hakim and Md. Shamsuzzoha Bayzid}

\ccsdesc[100]{Mathematics of computing~Graph theory}
\ccsdesc[100]{Theory of computation~Design and analysis of algorithms}

\keywords{Pairwise Compatibility Graphs, threshold graph representations, multi-interval PCGs, graph classes, tree-metric representations}

\begin{document}



\maketitle

\begin{abstract}
\textsc{Pairwise Compatibility Graphs} (\textsc{PCGs}) form a tree-metric graph class that originated in phylogeny and has since attracted sustained interest in graph theory. Several natural generalizations have been proposed in order to overcome the expressive limitations of classical \textsc{PCGs}, including \textsc{$k$-interval-\textsc{PCG}s}, \textsc{$k$-\textsc{OR}-\textsc{PCG}s}, and \textsc{$k$-\textsc{AND}-\textsc{PCG}s}. In this paper, we introduce \textsc{$(k,t)$-threshold-\textsc{PCG}s}, a threshold-based framework that unifies these generalized notions: adjacency is determined by whether at least $t$ among $k$ underlying \textsc{PCG} predicates accept the vertex pair. We investigate the expressive power of this model from both constructive and asymptotic viewpoints. On the positive side, we show that every graph on $n$ vertices is a \textsc{$(n,t)$-threshold-\textsc{PCG}} for every $1 \le t \le n$. On the negative side, we prove that for every fixed pair $(k,t)$, the class of \textsc{$(k,t)$-threshold-\textsc{PCG}s} is asymptotically rare among all graphs. As a consequence, we obtain sharp separations from previously studied models, including a strict expressive gap relative to \textsc{$k$-interval-\textsc{PCG}s}. We also study explicit obstruction families through incidence graphs and derive additional structural consequences for the conjunction case, including the strictness of the \textsc{$k$-\textsc{AND}-\textsc{PCG}} hierarchy and the failure of closure under complement.
\end{abstract}

\section{Introduction}

\textsc{Pairwise Compatibility Graphs} (\textsc{PCGs}) were introduced by Kearney, Munro, and Phillips in the context of phylogeny, where compatibility relations among taxa are modeled through distances in an underlying edge-weighted tree \cite{Kearney2003EfficientGOU}. Over time, however, the notion has developed into a graph-theoretic representation model of independent interest. Informally, a graph is a \textsc{PCG} if its vertices can be associated with the leaves of an edge-weighted tree so that two vertices are adjacent exactly when their leaf-to-leaf distance lies in a prescribed interval. This definition is easy to state and geometrically appealing, but the class it defines has proved to be surprisingly subtle.

A large body of work has been devoted to understanding the expressive power of \textsc{PCGs}. On the positive side, several nontrivial graph classes are known to admit \textsc{PCG} representations; for example, graphs of Dilworth number at most two are \textsc{PCGs} \cite{Calamoneri2013GraphsWDF,Calamoneri2014OnPCE}, triangle-free outerplanar $3$-graphs are \textsc{PCGs} \cite{Salma2012TriangleFreeO3AQ},  and further sufficient conditions are known for broader families \cite{Calamoneri2013OnTPAE,Hossain2016ANCAI}. On the negative side, the class is not universal: it was shown that not every graph is a \textsc{PCG} \cite{Yanhaona2010DiscoveringPCV}, and subsequent work developed explicit
obstruction families, proof techniques for non-\textsc{PCG}ness, and enumerative studies
of minimal non-\textsc{PCG} graphs
\cite{Durocher2013OnGTZ,Baiocchi2017SomeCOO,Baiocchi2018GraphsTAAO,Azam2020OnTEAJ}.
Surveys such as \cite{Rahman2020ASOA,Calamoneri2016PairwiseCGM,survey} document both the
breadth of the area and the persistence of several basic open problems. Even at small orders, the boundary is delicate: all graphs on at most seven vertices are \textsc{PCGs} \cite{Calamoneri2012AllGWN}, while more intricate behavior appears immediately beyond that range.

These limitations have motivated a number of natural generalizations. One direction relaxes the single-interval requirement while retaining a common witness tree, leading to \textsc{$k$-interval-\textsc{PCG}} \cite{Ahmed2017MultiintervalPCAP,Hayat2024AMTQ}. Another direction combines several \textsc{PCG} representations on the same vertex set by Boolean operations, giving rise to classes such as \textsc{$k$-\textsc{OR}-\textsc{PCG}} and \textsc{$k$-\textsc{AND}-\textsc{PCG}} \cite{Calamoneri2021OnGOI,survey}. These models enlarge the expressive scope of classical \textsc{PCGs}, but they have largely been studied in parallel rather than within a single framework. As a consequence, some basic comparative questions have remained poorly understood. For example, the recent literature asks for sharper separations between multi-interval and Boolean-combination models, for explicit obstruction families in low-level \textsc{OR}/\textsc{AND} classes, and for a better understanding of distinguished hard families from the multi-interval literature \cite{Calamoneri2021OnGOI,survey}.

In this paper, we unify these directions through a threshold-based framework. Given
$k$ underlying \textsc{PCG} predicates, we declare a pair of vertices adjacent whenever
at least $t$ of them accept the pair. This yields the class of
\textsc{$(k,t)$-threshold-\textsc{PCG}}s. The definition subsumes several previously
studied models as special cases: ordinary \textsc{PCGs} correspond to $(1,1)$,
\textsc{$k$-\textsc{OR}-\textsc{PCG}} is recovered when $t=1$, and
\textsc{$k$-\textsc{AND}-\textsc{PCG}} is recovered when $t=k$. In this sense,
\textsc{$(k,t)$-threshold-\textsc{PCG}}s provide a single framework for studying
fixed-size combinations of tree-distance graph representations.

This framework leads to a clear dichotomy. If the number of constituent predicates is
allowed to grow with the graph size, then the model becomes universal: every graph on
$n$ vertices is a \textsc{$(n,t)$-threshold-\textsc{PCG}} for every $1\le t\le n$. In contrast, if the number of predicates is fixed, then even after allowing all thresholds, the resulting classes occupy only an asymptotically vanishing fraction of all graphs. Thus the threshold model is both expressive and structurally restrictive. Moreover, it provides a convenient setting in which several phenomena can be studied uniformly: counting arguments yield separations from multi-interval representations, incidence graphs give explicit obstruction families, and the conjunction case exhibits both a strict hierarchy and failure of complement closure.

\paragraph*{Our contributions}
This paper initiates the systematic study of
\textsc{$(k,t)$-threshold-\textsc{PCG}}s. Our main contributions are as follows.
\begin{itemize}
    \item We prove a universality theorem: every graph on $n$ vertices is a
    \textsc{$(n,t)$-threshold-\textsc{PCG}} for every $1\le t\le n$.

    \item We prove an asymptotic sparsity theorem: for every fixed integer $k$, even
    after allowing all thresholds $t\in[k]$, the union of the corresponding threshold
    classes has asymptotic density zero among all graphs.

    \item We compare threshold combinations with multi-interval representations and show
    that \textsc{$(k,t)$-threshold-\textsc{PCG}}s can be substantially more expressive.
    In particular, if \(f(k)\in o(k^2/\log k)\), then for all sufficiently large \(k\)
    there exists a graph that is a \textsc{$(k,t)$-threshold-\textsc{PCG}} but not an
    \textsc{$f(k)$-interval-\textsc{PCG}}. As a consequence, for all sufficiently large
    \(k\),
    \[
    k\textsc{-interval-\textsc{PCG}} \subsetneq k\textsc{-OR-\textsc{PCG}}.
    \]

    \item We develop explicit obstruction families via set-system incidence graphs. This
    yields concrete lower bounds for threshold representations, explicit infinite
    families outside low-level threshold classes, and quantitative bounds for some graph classes.

    \item In the conjunction setting, we prove that the hierarchy of
    \textsc{$k$-\textsc{AND}-\textsc{PCG}} classes is strict:
    \[
    k\textsc{-AND-\textsc{PCG}} \subsetneq (k+1)\textsc{-AND-\textsc{PCG}}
    \qquad\text{for every }k\ge 1.
    \]

    \item We also show that the classes \(k\)-\textsc{AND}-\textsc{PCG} and k\textsc{-interval-\textsc{PCG}} are not closed under complement.
    
\end{itemize}

Taken together, these results show that threshold combinations provide a robust and
meaningful extension of the \textsc{PCG} paradigm. They are expressive enough to unify
and strengthen several known variants, yet rigid enough to admit strong counting
theorems, explicit obstructions, and nontrivial hierarchy results. We believe this makes
\textsc{$(k,t)$-threshold-\textsc{PCG}}s a natural object of study in their own right and
a useful organizing framework for future work on tree-distance graph representations.

\paragraph*{Organization}
Section~\ref{sec:preliminaries} introduces the relevant definitions and basic
relationships among the classes under consideration.
Section~\ref{sec:universality} proves universality when the number of predicates grows
with the graph size, and Section~\ref{sec:counting} proves asymptotic sparsity for fixed
\(k\).
Section~\ref{sec:consequences} compares the threshold model with multi-interval
representations.
Section~\ref{sec:incidence} develops explicit obstruction families via set-system
incidence graphs.
Section~\ref{sec:and} studies the hierarchy and complement behavior of
\textsc{$k$-\textsc{AND}-\textsc{PCG}}s.  For brevity, we defer the proof that
\textsc{$k$-interval-\textsc{PCG}} is not closed under complement to Appendix~\ref{app:kint}. We conclude in Section~\ref{sec:conclusion} with some open problems. 

\section{Preliminaries and Definitions}\label{sec:preliminaries}

We begin by recalling the standard definition of \textsc{Pairwise Compatibility Graphs} and the main generalizations relevant to this paper. We then introduce \textsc{$(k,t)$-threshold-\textsc{PCG}s} formally and record the basic identities that place \textsc{PCG}, \textsc{$k$-\textsc{OR}-\textsc{PCG}}, and \textsc{$k$-\textsc{AND}-\textsc{PCG}} inside a single threshold framework.

Throughout the paper, all graphs are finite, simple, and undirected. For a positive integer $k$, we write
\(
[k]:=\{1,2,\dots,k\}.
\)
If $T$ is an edge-weighted tree and $x,y$ are leaves of $T$, then $d_T(x,y)$ denotes the sum of the edge weights along the unique $x$--$y$ path in $T$. We also denote the leafset of $T$ as $L(T)$. Furthermore, we call a tree $T$ \textit{reduced} if it has no vertices of degree 2.

We first recall the classical notion of a \textsc{Pairwise Compatibility Graph}, introduced in \cite{Kearney2003EfficientGOU}; see also \cite{Rahman2020ASOA,Calamoneri2016PairwiseCGM,survey}.

\begin{definition}[\textsc{Pairwise Compatibility Graph \cite{Kearney2003EfficientGOU}}]
Let $G=(V,E)$ be a graph. We say that $G$ is a \textsc{PCG} if there exist an edge-weighted tree $T$, two nonnegative real numbers $d_{\min},d_{\max}$ with $d_{\min}\le d_{\max}$, and a bijection
\(
\zeta:V\to L(T),
\)
where $L(T)$ denotes the leaf set of $T$, such that for every two distinct vertices $u,v\in V$,
\[
uv\in E
\quad\Longleftrightarrow\quad
d_{\min}\le d_T(\zeta(u),\zeta(v))\le d_{\max}.
\]
In this case, we say that $(T,d_{\min},d_{\max},\zeta)$ is a \textsc{PCG} representation of $G$.

When the bijection $\zeta$ is clear from the context, we suppress it from the notation and simply write
\(
G=\textsc{PCG}(T,d_{\min},d_{\max}).
\) or \(
G=\textsc{PCG}(T,[d_{\min},d_{\max}]).
\)
\end{definition}

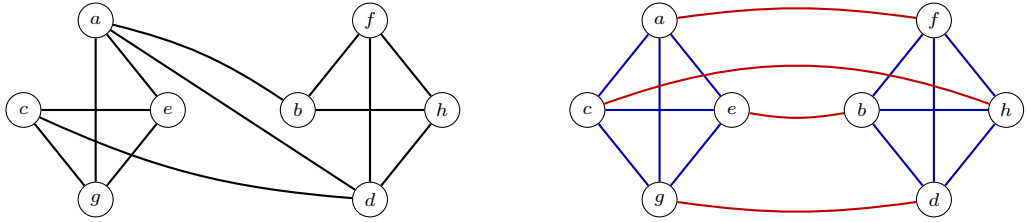
\begin{figure}[t]
\centering

\begin{minipage}{\textwidth}
\centering
\begin{tikzpicture}[
    x=0.95cm,y=0.95cm,
    every node/.style={font=\small},
    int/.style={circle,draw,fill=white,minimum size=4.4mm,inner sep=0pt},
    leaf/.style={circle,draw,fill=white,minimum size=4.4mm,inner sep=0pt},
    wlbl/.style={font=\scriptsize,fill=white,inner sep=1pt}
]

\node at (0,1.0) {\textbf{(a)} A weighted tree \(T\) on eight leaves};

\node[int] (r) at (0,0) {};
\node[int] (L) at (-2.2,-1.4) {};
\node[int] (R) at ( 2.2,-1.4) {};

\draw (r) -- node[wlbl,above left] {$10$} (L);
\draw (r) -- node[wlbl,above right] {$10$} (R);

\node[leaf] (a) at (-4.1,-2.8) {$a$};
\node[leaf] (c) at (-2.9,-3.5) {$c$};
\node[leaf] (e) at (-1.5,-3.5) {$e$};
\node[leaf] (g) at (-0.3,-2.8) {$g$};

\draw (L) -- node[wlbl,left] {$1$} (a);
\draw (L) -- node[wlbl,left] {$2$} (c);
\draw (L) -- node[wlbl,right] {$3$} (e);
\draw (L) -- node[wlbl,right] {$4$} (g);

\node[leaf] (f) at ( 0.3,-2.8) {$f$};
\node[leaf] (h) at ( 1.5,-3.5) {$h$};
\node[leaf] (b) at ( 2.9,-3.5) {$b$};
\node[leaf] (d) at ( 4.1,-2.8) {$d$};

\draw (R) -- node[wlbl,left] {$4$} (f);
\draw (R) -- node[wlbl,left] {$3$} (h);
\draw (R) -- node[wlbl,right] {$2$} (b);
\draw (R) -- node[wlbl,right] {$1$} (d);

\end{tikzpicture}
\end{minipage}

\vspace{1em}

\begin{minipage}{0.47\textwidth}
\centering
\begin{tikzpicture}[
    x=0.95cm,y=0.95cm,
    v/.style={circle,draw,fill=white,minimum size=4.6mm,inner sep=0pt,font=\scriptsize},
    edgepcg/.style={draw=black,thick}
]

\node at (0,2.55) {\textbf{(b)} \(\mathrm{PCG}(T,[4,23])\)};

\node[v] (a) at (-1.9, 1.4) {$a$};
\node[v] (c) at (-2.9, 0.15) {$c$};
\node[v] (e) at (-0.9, 0.15) {$e$};
\node[v] (g) at (-1.9,-1.1) {$g$};

\node[v] (f) at ( 1.9, 1.4) {$f$};
\node[v] (h) at ( 2.9, 0.15) {$h$};
\node[v] (b) at ( 0.9, 0.15) {$b$};
\node[v] (d) at ( 1.9,-1.1) {$d$};

\draw[edgepcg] (a) -- (e);
\draw[edgepcg] (a) -- (g);
\draw[edgepcg] (c) -- (e);
\draw[edgepcg] (c) -- (g);
\draw[edgepcg] (e) -- (g);

\draw[edgepcg] (f) -- (h);
\draw[edgepcg] (f) -- (b);
\draw[edgepcg] (f) -- (d);
\draw[edgepcg] (h) -- (b);
\draw[edgepcg] (h) -- (d);

\draw[edgepcg] (a) to[bend right=0] (d);
\draw[edgepcg] (a) to[bend left=10] (b);
\draw[edgepcg] (c) to[bend right=10] (d);

\end{tikzpicture}
\end{minipage}
\hfill
\begin{minipage}{0.47\textwidth}
\centering
\begin{tikzpicture}[
    x=0.95cm,y=0.95cm,
    v/.style={circle,draw,fill=white,minimum size=4.6mm,inner sep=0pt,font=\scriptsize},
    edgeone/.style={draw=blue!70!black,thick},
    edgetwo/.style={draw=red!75!black,thick}
]

\node at (0,2.55) {\textbf{(c)} \(\mathrm{2\text{-}interval\text{-}PCG}(T,{\color{blue}[3,7]}\cup{\color{red}[25,25]})\)};

\node[v] (a) at (-1.9, 1.4) {$a$};
\node[v] (c) at (-2.9, 0.15) {$c$};
\node[v] (e) at (-0.9, 0.15) {$e$};
\node[v] (g) at (-1.9,-1.1) {$g$};

\node[v] (f) at ( 1.9, 1.4) {$f$};
\node[v] (h) at ( 2.9, 0.15) {$h$};
\node[v] (b) at ( 0.9, 0.15) {$b$};
\node[v] (d) at ( 1.9,-1.1) {$d$};

\draw[edgeone] (a)--(c);
\draw[edgeone] (a)--(e);
\draw[edgeone] (a)--(g);
\draw[edgeone] (c)--(e);
\draw[edgeone] (c)--(g);
\draw[edgeone] (e)--(g);

\draw[edgeone] (b)--(d);
\draw[edgeone] (b)--(f);
\draw[edgeone] (b)--(h);
\draw[edgeone] (d)--(f);
\draw[edgeone] (d)--(h);
\draw[edgeone] (f)--(h);

\draw[edgetwo] (a) to[bend left=8] (f);
\draw[edgetwo] (c) to[bend left=20] (h);
\draw[edgetwo] (e) to[bend right=10] (b);
\draw[edgetwo] (g) to[bend right=8] (d);



\end{tikzpicture}
\end{minipage}

\caption{Panel (a) shows an edge-weighted tree. Panel (b) shows a graph that is a PCG for the given tree and the interval [4, 23]. Panel (c) shows a graph that is a \textsc{$2$-Interval-\textsc{PCG}} for the given tree and intervals [3, 7] and [25, 25]. The graph in Panel (c) is one of the smallest known graphs that is not a PCG. \cite{ Azam2020AMFAN, Calamoneri2022AllGWAA}}
\label{fig:pcg-2interval-overview}
\end{figure}

A natural extension is to allow several admissible distance intervals while keeping a single witness tree. This yields the notion of \textsc{$k$-interval-\textsc{PCG}} \cite{Ahmed2017MultiintervalPCAP}; see also \cite{Hayat2024AMTQ,Papan2023On2PA,Hakim2022NewROR}.

\begin{definition}[\textsc{$k$-Interval-\textsc{PCG} \cite{Ahmed2017MultiintervalPCAP}}]
Let $k$ be a positive integer. A graph $G=(V,E)$ is a \textsc{$k$-interval-\textsc{PCG}} if there exist an edge-weighted tree $T$, intervals
\(
I_1,\dots,I_k
\)
of nonnegative real numbers, and a bijection
\(
\zeta:V\to L(T)
\)
such that, for every two distinct vertices $u,v\in V$,
\[
uv\in E
\quad\Longleftrightarrow\quad
d_T(\zeta(u),\zeta(v))\in \bigcup_{i=1}^k I_i.
\]
\end{definition}

By merging overlapping intervals if necessary, one may always assume that $I_1,\dots,I_k$ are pairwise disjoint.

A different line of generalization combines several \textsc{PCG} representations on the same vertex set by Boolean operations. In particular, union and intersection variants were studied in \cite{Calamoneri2021OnGOI,survey}.

\begin{definition}[\textsc{$k$-\textsc{OR}-\textsc{PCG}} \cite{Calamoneri2021OnGOI}]
Let $k$ be a positive integer. A graph $G=(V,E)$ is a \textsc{$k$-\textsc{OR}-\textsc{PCG}} if there exist \textsc{PCGs}
\[
G_1=(V,E_1),G_2=(V,E_2),\ \dots\ ,G_k=(V,E_k)
\]
on the same vertex set $V$ such that
\[
E=E_1\cup E_2\cup\cdots\cup E_k.
\]
\end{definition}

\begin{definition}[\textsc{$k$-\textsc{AND}-\textsc{PCG}} \cite{Calamoneri2021OnGOI}]
Let $k$ be a positive integer. A graph $G=(V,E)$ is a \textsc{$k$-\textsc{AND}-\textsc{PCG}} if there exist \textsc{PCGs}
\[
G_1=(V,E_1),G_2=(V,E_2),\ \dots\ ,G_k=(V,E_k)
\]
on the same vertex set $V$ such that
\[
E=E_1\cap E_2\cap\cdots\cap E_k.
\]
\end{definition}
\begin{figure}[t]
\centering
\resizebox{\textwidth}{!}{%
\begin{tikzpicture}[
    line cap=round,
    line join=round,
    panel/.style={rounded corners=4pt, draw=black!40},
    leaf/.style={circle, draw, fill=white, minimum size=4.6mm, inner sep=0pt, font=\scriptsize},
    inode/.style={circle, fill=black, minimum size=1.8mm, inner sep=0pt},
    edge/.style={draw=black, semithick},
    common/.style={draw=black, semithick},
    h1only/.style={draw=blue!70!black, thick},
    h2only/.style={draw=red!75!black, thick, dashed},
    wlbl/.style={font=\small, fill=white, inner sep=0.5pt},
    ptitle/.style={font=\small\bfseries},
    pnote/.style={font=\scriptsize}
]


\node[ptitle] at (4.0, 14.5) {(a) $H_1$};
\begin{scope}[shift={(4.0, 12.2)}]
\node[leaf] (a) at ( 1.85, 0.00) {$a$};
\node[leaf] (b) at ( 1.30, 1.30) {$b$};
\node[leaf] (c) at ( 0.00, 1.85) {$c$};
\node[leaf] (d) at (-1.30, 1.30) {$d$};
\node[leaf] (e) at (-1.85, 0.00) {$e$};
\node[leaf] (f) at (-1.30,-1.30) {$f$};
\node[leaf] (g) at ( 0.00,-1.85) {$g$};
\node[leaf] (h) at ( 1.30,-1.30) {$h$};

\draw[edge] (a)--(c); \draw[edge] (a)--(e);
\draw[edge] (a)--(f); \draw[edge] (a)--(g);
\draw[edge] (b)--(d); \draw[edge] (b)--(e);
\draw[edge] (b)--(f); \draw[edge] (b)--(h);
\draw[edge] (c)--(e); \draw[edge] (c)--(g);
\draw[edge] (c)--(h); \draw[edge] (d)--(f);
\draw[edge] (d)--(g); \draw[edge] (d)--(h);
\draw[edge] (e)--(g); \draw[edge] (f)--(h);

\draw[edge] (a) to[bend right=12] (h);
\draw[edge] (b) to[bend left=12] (g);
\end{scope}

\node[ptitle] at (12.0, 14.5) {(b) Witness tree for $H_1$ with interval $[12,24]$};
\begin{scope}[shift={(12.0, 12.2)}]
\node[inode] (s5) at (-2.5, 0.0) {};
\node[inode] (s4) at (-1.5, 0.0) {};
\node[inode] (s3) at (-0.5, 0.0) {};
\node[inode] (s2) at ( 0.5, 0.0) {};
\node[inode] (s1) at ( 1.5, 0.0) {};
\node[inode] (r)  at ( 2.5, 0.0) {};

\draw (s5) -- node[wlbl,above] {$1$} (s4);
\draw (s4) -- node[wlbl,above] {$3$} (s3);
\draw (s3) -- node[wlbl,above] {$6$} (s2);
\draw (s2) -- node[wlbl,above] {$1$} (s1);
\draw (s1) -- node[wlbl,above] {$1$} (r);

\node[leaf] (g) at (-2.5, 0.8) {$g$};
\node[leaf] (d) at (-2.5,-0.8) {$d$};
\node[leaf] (f) at (-1.5,-0.8) {$f$};
\node[leaf] (h) at (-0.5, 0.8) {$h$};
\node[leaf] (c) at ( 0.5,-0.8) {$c$};
\node[leaf] (a) at ( 1.5, 0.8) {$a$};
\node[leaf] (b) at ( 2.5, 0.8) {$b$};
\node[leaf] (e) at ( 2.5,-0.8) {$e$};

\draw (s5) -- node[wlbl,left]  {$0$}  (g);
\draw (s5) -- node[wlbl,left]  {$12$} (d);
\draw (s4) -- node[wlbl,left]  {$9$}  (f);
\draw (s3) -- node[wlbl,left]  {$6$}  (h);
\draw (s2) -- node[wlbl,left]  {$8$}  (c);
\draw (s1) -- node[wlbl,left]  {$5$}  (a);
\draw (r)  -- node[wlbl,right] {$0$}  (b);
\draw (r)  -- node[wlbl,right] {$12$} (e);
\end{scope}


\node[ptitle] at (4.0, 9.5) {(c) $H_2$};
\begin{scope}[shift={(4.0, 7.2)}]
\node[leaf] (a) at ( 1.85, 0.00) {$a$};
\node[leaf] (b) at ( 1.30, 1.30) {$b$};
\node[leaf] (c) at ( 0.00, 1.85) {$c$};
\node[leaf] (d) at (-1.30, 1.30) {$d$};
\node[leaf] (e) at (-1.85, 0.00) {$e$};
\node[leaf] (f) at (-1.30,-1.30) {$f$};
\node[leaf] (g) at ( 0.00,-1.85) {$g$};
\node[leaf] (h) at ( 1.30,-1.30) {$h$};

\draw[edge] (a)--(c); \draw[edge] (a)--(e);
\draw[edge] (a)--(f); \draw[edge] (a)--(g);
\draw[edge] (b)--(d); \draw[edge] (b)--(e);
\draw[edge] (b)--(f); \draw[edge] (b)--(h);
\draw[edge] (c)--(e); \draw[edge] (c)--(g);
\draw[edge] (c)--(h); \draw[edge] (d)--(f);
\draw[edge] (d)--(g); \draw[edge] (d)--(h);
\draw[edge] (e)--(g); \draw[edge] (f)--(h);

\draw[edge] (c) to[bend left=12] (f);
\draw[edge] (d) to[bend right=12] (e);
\end{scope}

\node[ptitle] at (12.0, 9.5) {(d) Witness tree for $H_2$ with interval $[10,20]$};
\begin{scope}[shift={(12.0, 6.6)}]
\node[inode] (s3) at (-1.8, 0.0) {};
\node[inode] (s2) at (-0.6, 0.0) {};
\node[inode] (s1) at ( 0.6, 0.0) {};
\node[inode] (r)  at ( 1.8, 0.0) {};
\node[inode] (t1) at ( 0.6, 1.0) {}; 

\draw (s3) -- node[wlbl,above] {$1$} (s2);
\draw (s2) -- node[wlbl,above] {$4$} (s1);
\draw (s1) -- node[wlbl,above] {$4$} (r);
\draw (s1) -- node[wlbl,left]  {$3$} (t1);

\node[leaf] (d) at (-2.6, 0.6) {$d$};
\node[leaf] (c) at (-2.6,-0.6) {$c$};
\node[leaf] (a) at (-0.6,-0.9) {$a$};
\node[leaf] (g) at (-0.2, 1.8) {$g$};
\node[leaf] (f) at ( 1.4, 1.8) {$f$};
\node[leaf] (b) at ( 2.6, 0.8) {$b$};
\node[leaf] (e) at ( 2.8, 0.0) {$e$};
\node[leaf] (h) at ( 2.6,-0.8) {$h$};

\draw (s3) -- node[wlbl,above right]{$1$} (d);
\draw (s3) -- node[wlbl,below right]{$3$} (c);
\draw (s2) -- node[wlbl,left]       {$6$} (a);
\draw (t1) -- node[wlbl,above left] {$7$} (g);
\draw (t1) -- node[wlbl,above right]{$1$} (f);
\draw (r)  -- node[wlbl,above left] {$10$}(b);
\draw (r)  -- node[wlbl,above]      {$0$} (e);
\draw (r)  -- node[wlbl,below left] {$8$} (h);
\end{scope}


\node[ptitle] at (4.0, 4.5) {(e) $A=H_1\cup H_2$ is a \textsc{2-OR-PCG}};
\begin{scope}[shift={(4.0, 2.2)}]
\node[leaf] (a) at ( 1.85, 0.00) {$a$};
\node[leaf] (b) at ( 1.30, 1.30) {$b$};
\node[leaf] (c) at ( 0.00, 1.85) {$c$};
\node[leaf] (d) at (-1.30, 1.30) {$d$};
\node[leaf] (e) at (-1.85, 0.00) {$e$};
\node[leaf] (f) at (-1.30,-1.30) {$f$};
\node[leaf] (g) at ( 0.00,-1.85) {$g$};
\node[leaf] (h) at ( 1.30,-1.30) {$h$};

\draw[common] (a)--(c); \draw[common] (a)--(e);
\draw[common] (a)--(f); \draw[common] (a)--(g);
\draw[common] (b)--(d); \draw[common] (b)--(e);
\draw[common] (b)--(f); \draw[common] (b)--(h);
\draw[common] (c)--(e); \draw[common] (c)--(g);
\draw[common] (c)--(h); \draw[common] (d)--(f);
\draw[common] (d)--(g); \draw[common] (d)--(h);
\draw[common] (e)--(g); \draw[common] (f)--(h);

\draw[h1only] (a) to[bend right=12] (h);
\draw[h1only] (b) to[bend left=12] (g);

\draw[h2only] (c) to[bend left=12] (f);
\draw[h2only] (d) to[bend right=12] (e);
\end{scope}

\node[pnote] at (4.0, -0.2) 
{black: in both $H_1,H_2$;\quad blue: only in $H_1$;\quad red dashed: only in $H_2$};

\node[ptitle] at (12.0, 4.5) {(f) $B=H_1\cap H_2$ is a \textsc{2-AND-PCG}};
\begin{scope}[shift={(12.0, 2.2)}]
\node[leaf] (a) at ( 1.85, 0.00) {$a$};
\node[leaf] (b) at ( 1.30, 1.30) {$b$};
\node[leaf] (c) at ( 0.00, 1.85) {$c$};
\node[leaf] (d) at (-1.30, 1.30) {$d$};
\node[leaf] (e) at (-1.85, 0.00) {$e$};
\node[leaf] (f) at (-1.30,-1.30) {$f$};
\node[leaf] (g) at ( 0.00,-1.85) {$g$};
\node[leaf] (h) at ( 1.30,-1.30) {$h$};

\draw[common] (a)--(c); \draw[common] (a)--(e);
\draw[common] (a)--(f); \draw[common] (a)--(g);
\draw[common] (b)--(d); \draw[common] (b)--(e);
\draw[common] (b)--(f); \draw[common] (b)--(h);
\draw[common] (c)--(e); \draw[common] (c)--(g);
\draw[common] (c)--(h); \draw[common] (d)--(f);
\draw[common] (d)--(g); \draw[common] (d)--(h);
\draw[common] (e)--(g); \draw[common] (f)--(h);
\end{scope}

\end{tikzpicture}%
}
\caption{Two \textsc{PCG} witnesses $H_1$ and $H_2$ on a shared vertex set, together with their union and intersection. Consequently, $A=H_1\cup H_2$ is a \textsc{2-OR-PCG} or equivalently, a \textsc{$(2,1)$-Threshold-\textsc{PCG}}, while $B=H_1\cap H_2$ is a \textsc{2-AND-PCG}, or equivalently, a \textsc{$(2,2)$-Threshold-\textsc{PCG}}. Both $A$ and $B$ are known not to be PCGs \cite{Azam2020AMFAN,Calamoneri2022AllGWAA}.}
\label{fig:prelim-pcg-or-and}
\end{figure}

We now define a new graph class, \textsc{$(k,t)$-Threshold-\textsc{PCG}}, as a generalization for the \textsc{$k$-\textsc{OR}-\textsc{PCGs}} and \textsc{$k$-\textsc{AND}-\textsc{PCGs}}.

\begin{definition}[\textsc{$(k,t)$-Threshold-\textsc{PCG}}]
Let $k$ and $t$ be integers with $1\le t\le k$. A graph $G=(V,E)$ is a \textsc{$(k,t)$-threshold-\textsc{PCG}} if there exist \textsc{PCGs}
\[
G_1=(V,E_1),G_2=(V,E_2),\ \dots\ ,G_k=(V,E_k)
\]
on the same vertex set $V$ such that, for every two distinct vertices $u,v\in V$,
\[
uv\in E
\quad\Longleftrightarrow\quad
\bigl|\{\,i\in[k]: uv\in E_i\,\}\bigr|\ge t.
\]
\end{definition}

Thus, a pair of vertices is adjacent precisely when it is accepted by at least $t$ among the $k$ underlying \textsc{PCG} predicates.

The next proposition records the three basic special cases that justify our threshold viewpoint.

\begin{proposition}\label{prop:special-cases}
Let $k\ge 1$. Then:
\begin{enumerate}
    \item a graph is a \textsc{PCG} if and only if it is a \textsc{$(1,1)$-threshold-\textsc{PCG}};
    \item a graph is a \textsc{$k$-\textsc{OR}-\textsc{PCG}} if and only if it is a \textsc{$(k,1)$-threshold-\textsc{PCG}};
    \item a graph is a \textsc{$k$-\textsc{AND}-\textsc{PCG}} if and only if it is a \textsc{$(k,k)$-threshold-\textsc{PCG}}.
\end{enumerate}
\end{proposition}

\begin{proof}
The first statement is immediate from the definitions. For the second statement, a pair $uv$ is adjacent in a \textsc{$(k,1)$-threshold-\textsc{PCG}} if and only if it belongs to at least one among $E_1,\dots,E_k$, that is, if and only if
\(
uv\in E_1\cup\cdots\cup E_k.
\)
Hence this is exactly \textsc{$k$-\textsc{OR}-\textsc{PCG}}. Similarly, a pair $uv$ is adjacent in a \textsc{$(k,k)$-threshold-\textsc{PCG}} if and only if it belongs to all of $E_1,\dots,E_k$, that is, if and only if
\(
uv\in E_1\cap\cdots\cap E_k.
\)
Thus this is exactly \textsc{$k$-\textsc{AND}-\textsc{PCG}}.
\end{proof}

Proposition~\ref{prop:special-cases} shows that \textsc{$(k,t)$-threshold-\textsc{PCG}}s provide a single umbrella under which several previously studied models can be treated uniformly.

\section{Universality When the Number of Predicates Grows}\label{sec:universality}

We now show that the threshold model becomes universal once the number of underlying \textsc{PCG} predicates is allowed to grow with the graph size.

We first record two simple padding lemmas. Since both $K_n$ and $\overline{K_n}$ are trivially \textsc{PCGs}, they may be used as auxiliary predicates in threshold constructions.

\begin{lemma}\label{lem:or-to-threshold}
Let $G$ be a graph on $n$ vertices. If $G$ is a \textsc{$q$-\textsc{OR}-\textsc{PCG}}, then $G$ is a \textsc{$(n,t)$-threshold-\textsc{PCG}} for every
\(
1 \le t \le n-q+1.
\)
\end{lemma}

\begin{proof}
Since $G$ is a \textsc{$q$-\textsc{OR}-\textsc{PCG}}, write 
\(
G = G_1 \cup \cdots \cup G_q,
\)
where each $G_i$ is a \textsc{PCG}. Add $t-1$ copies of $K_n$ and $n-q-(t-1)$ copies of $\overline{K_n}$. Then every edge of $G$ is accepted by at least one of $G_1,\dots,G_q$ and by all $t-1$ copies of $K_n$, hence by at least $t$ predicates in total. Every nonedge is accepted by none of $G_1,\dots,G_q$ and only by the $t-1$ copies of $K_n$, hence by exactly $t-1$ predicates. Therefore, $G$ is a \textsc{$(n,t)$-threshold-\textsc{PCG}}.
\end{proof}

\begin{lemma}\label{lem:and-to-threshold}
Let $G$ be a graph on $n$ vertices. If $G$ is a \textsc{$q$-\textsc{AND}-\textsc{PCG}}, then $G$ is a \textsc{$(n,t)$-threshold-\textsc{PCG}} for every
\(
q \le t \le n.
\)
\end{lemma}

\begin{proof}
Since $G$ is a \textsc{$q$-\textsc{AND}-\textsc{PCG}}, write
\(
G = G_1 \cap \cdots \cap G_q,
\)
where each $G_i$ is a \textsc{PCG}. Add $t-q$ copies of $K_n$ and $n-t$ copies of $\overline{K_n}$. Then every edge of $G$ is accepted by all $q$ original predicates and by all $t-q$ copies of $K_n$, hence by at least $t$ predicates. Every nonedge is missing from at least one of $G_1,\dots,G_q$, so it is accepted by at most $(q-1)+(t-q)=t-1$ predicates. Therefore, $G$ is a \textsc{$(n,t)$-threshold-\textsc{PCG}}.
\end{proof}

\begin{theorem}\label{thm:universality}
For every integer $n \ge 1$ and every $t \in [n]$, every graph on $n$ vertices is a \textsc{$(n,t)$-threshold-\textsc{PCG}}.
\end{theorem}

\begin{proof}
Let $G$ be any graph on $n$ vertices.

If $n\le 7$, then $G$ is a \textsc{PCG} by \cite{Calamoneri2012AllGWN}, hence a \textsc{$1$-\textsc{OR}-\textsc{PCG}}. Lemma~\ref{lem:or-to-threshold} with $q=1$ gives the claim for all $t\in[n]$.

Assume now that $n\ge 8$. By the known \textsc{OR} bound \cite{Calamoneri2021OnGOI}, every graph on $n$ vertices is a \textsc{$q_{\mathrm{OR}}$-\textsc{OR}-\textsc{PCG}} with
\(
q_{\mathrm{OR}}=\left\lceil \frac{1}{3}(n-7)\right\rceil+1.
\)
Hence Lemma~\ref{lem:or-to-threshold} shows that $G$ is a \textsc{$(n,t)$-threshold-\textsc{PCG}} for every
\(
1 \le t \le n-\left\lceil \frac{1}{3}(n-7)\right\rceil.
\)

Also, by the known \textsc{AND} bound \cite{Calamoneri2021OnGOI}, every graph on $n$ vertices is a \textsc{$q_{\mathrm{AND}}$-\textsc{AND}-\textsc{PCG}} with
\(
q_{\mathrm{AND}}=\left\lfloor \frac{n}{2}\right\rfloor.
\)
Hence Lemma~\ref{lem:and-to-threshold} shows that $G$ is a \textsc{$(n,t)$-threshold-\textsc{PCG}} for every
\(
\left\lfloor \frac{n}{2}\right\rfloor \le t \le n.
\)

These two ranges cover all of $[n]$, since
\(
n-\left\lceil \frac{1}{3}(n-7)\right\rceil \ge \left\lfloor \frac{n}{2}\right\rfloor.
\)
Therefore, $G$ is a \textsc{$(n,t)$-threshold-\textsc{PCG}} for every $t\in[n]$.
\end{proof}

Theorem~\ref{thm:universality} shows that the threshold framework becomes fully expressive once the number of predicates is allowed to grow linearly with the number of vertices.

\section{Asymptotic Sparsity for Fixed $k$}\label{sec:counting}

We now consider the opposite regime, where the number of underlying predicates is fixed. Even if the threshold parameter is allowed to vary over all values in $[k]$, the resulting graph classes remain asymptotically negligible.

\begin{lemma}\label{lem:pcg-count}
There exists an absolute constant $C>0$ such that the number $P_n$ of labeled
\textsc{PCGs} on vertex set $[n]$ satisfies $P_n \le e^{C n\log n}$.
\end{lemma}

\begin{proof}
Let $G$ be a labeled \textsc{PCG} on vertex set $[n]$. By suppressing degree-$2$
internal vertices in a witness tree, we may assume that $G=\textsc{PCG}(T,d_{\min},d_{\max})$
for some reduced weighted tree $T$ with leaf set $[n]$. Such a tree has at most $2n-3$
edges, and the number of leaf-labeled reduced binary tree topologies on $[n]$ is
$(2n-5)!!=e^{O(n\log n)}$ \cite{Felsenstein1978,SempleSteel2003}.

Fix one such topology $T$, and let $w_1,\dots,w_m$ be its edge weights, where
$m\le 2n-3$. For each pair of leaves $u,v\in[n]$, the distance $d_T(u,v)$ is a linear
form in the variables $w_1,\dots,w_m$. Hence, for fixed $T$, the represented graph is
determined by the sign pattern of the $O(n^2)$ linear polynomials
$d_T(u,v)-d_{\min}$ and $d_T(u,v)-d_{\max}$ in the $m+2=O(n)$ variables
$w_1,\dots,w_m,d_{\min},d_{\max}$. By Theorem 3 of \cite{Warren1968}, the number of such sign
patterns is at most $e^{O(n\log n)}$.

Multiplying by the $e^{O(n\log n)}$ possible reduced leaf-labeled tree topologies gives
$P_n \le e^{O(n\log n)}$, as required.
\end{proof}
\begin{theorem}\label{thm:sparse-uniform-k}
Let $k \ge 1$ be fixed. Then there exists an integer $n_0$ such that, for every
$n \ge n_0$, there exists a graph on $n$ vertices that is neither a
$(k,t)$-threshold-\textsc{PCG} nor the complement of a
$(k,t)$-threshold-\textsc{PCG}, for every $t \in [k]$.
\end{theorem}

\begin{proof}
Let $P_n$ denote the number of labeled \textsc{PCGs} on vertex set $[n]$.
By Lemma \ref{lem:pcg-count},
\(
P_n \le e^{Cn\log n}.
\)

For fixed $n,k,t$, let $T_{n,k,t}$ denote the number of labeled
\textsc{$(k,t)$-threshold-\textsc{PCG}}s on $[n]$. Any ordered $k$-tuple of labeled
\textsc{PCGs} $(G_1,\dots,G_k)$ on $[n]$ determines, via the threshold rule, at most one
labeled \textsc{$(k,t)$-threshold-\textsc{PCG}}. Since there are $P_n$ choices for each
$G_i$, there are at most $(P_n)^k$ such tuples. Hence
\[
T_{n,k,t}\le (P_n)^k \le e^{Ck\,n\log n}.
\]

Let $\mathcal{F}_{n,k}$ be the family of labeled graphs on $[n]$ that are
$(k,t)$-threshold-\textsc{PCGs} for some $t\in[k]$, together with their complements.
Then
\[
|\mathcal{F}_{n,k}|
\le \sum_{t=1}^k 2T_{n,k,t}
\le 2k\,e^{Ck\,n\log n}.
\]

On the other hand, the total number of labeled graphs on $[n]$ is
\(
2^{\binom{n}{2}} = e^{\Theta(n^2)}.
\)
Since $k$ is fixed,
\(
2k\,e^{Ck\,n\log n}=o\!\left(2^{\binom{n}{2}}\right).
\)
Hence, for all sufficiently large $n$,
\(
2k\,e^{Ck\,n\log n}<2^{\binom{n}{2}}.
\)
Therefore, some labeled graph on $[n]$ lies outside $\mathcal{F}_{n,k}$.
By definition, this graph is neither a $(k,t)$-threshold-\textsc{PCG} nor the complement of one, for every $t\in[k]$.
\end{proof}

It is convenient to record the counting consequence explicitly.

\begin{corollary}\label{cor:density-zero-uniform}
Let $k\ge 1$ be fixed, and let $\mathcal{T}_{n,k}$ denote the family of labeled graphs on $[n]$ that are $(k,t)$-threshold-\textsc{PCGs} for some $t\in[k]$. Then
\[
\frac{|\mathcal{T}_{n,k}|}{2^{\binom{n}{2}}}\to 0
\qquad\text{as } n\to\infty.
\]
In particular, the union over all thresholds $t\in[k]$ still has asymptotic density zero.
\end{corollary}

\begin{proof}
By the proof above,
\[
|\mathcal{T}_{n,k}|
\le \sum_{t=1}^k T_{n,k,t}
\le k\,e^{Ck\,n\log n}.
\]
Since $2^{\binom{n}{2}}=e^{\Theta(n^2)}$ and $k$ is fixed, the ratio tends to zero.
\end{proof}

Theorem~\ref{thm:sparse-uniform-k} and Corollary~\ref{cor:density-zero-uniform}
stand in sharp contrast both with Section~\ref{sec:universality} and with the early
optimism surrounding \textsc{PCGs}. While the model was introduced by Kearney, Munro,
and Phillips \cite{Kearney2003EfficientGOU}, it was later conjectured that every graph
might be a \textsc{PCG}, a possibility subsequently disproved in
\cite{Yanhaona2010DiscoveringPCV}. Our result shows that every fixed threshold extension remains
far from universal in a much stronger asymptotic sense: for fixed $k$, even after allowing
all thresholds $t\in[k]$ and even up to complement, these classes occupy only an
asymptotically vanishing corner of the graph universe.

\section{Separating Multi-interval and Threshold Models}\label{sec:consequences}

We now examine more closely the relationship between threshold combinations and
multi-interval representations. These two frameworks arise from different ways of
relaxing the classical \textsc{PCG} model. In a \textsc{$k$-interval-\textsc{PCG}},
one keeps a single witness tree and increases expressive power by allowing several
admissible distance intervals on that tree \cite{Ahmed2017MultiintervalPCAP}. In a
\textsc{$(k,t)$-threshold-\textsc{PCG}}, by contrast, one allows several independent
\textsc{PCG} predicates and combines them through a threshold rule. Since both are natural extensions of \textsc{PCG}, it is
important to compare their expressive strength. The next theorem shows that, from an
asymptotic viewpoint, threshold combinations can be substantially more powerful than
multi-interval representations.

\begin{theorem}\label{thm:threshold-vs-interval} Let $f(k)\in o\!\left(\frac{k^2}{\log k}\right)$, and let $t$ be an integer such that $1\le t\le k$. Then there exists an integer $k_0$ such that, for every $k\ge k_0$, there exists a graph that is a \textsc{$(k,t)$-threshold-\textsc{PCG}} but not an \textsc{$f(k)$-interval-\textsc{PCG}}. \end{theorem}

\begin{proof}
Let $I_k^{(f)}$ denote the number of labeled \textsc{$f(k)$-interval-\textsc{PCG}}s on
vertex set $[k]$. By Theorem~\ref{thm:universality}, every graph on $k$ vertices is a
\textsc{$(k,t)$-threshold-\textsc{PCG}}. Thus it suffices to show that
\(
I_k^{(f)} < 2^{\binom{k}{2}}
\)
for all sufficiently large $k$.

Fix a labeled \textsc{$f(k)$-interval-\textsc{PCG}} on $[k]$. After suppressing
degree-$2$ internal vertices in a witness tree, we may assume that its representation
uses a reduced weighted tree with leaf set $[k]$ and at most $2k-3$ edges. The number of
such leaf-labeled tree topologies is $e^{O(k\log k)}$.

Now fix one such topology. Its representation is determined by at most $2k-3$ edge
weights together with at most $2f(k)$ interval endpoints, so the relevant parameter space
has dimension
\(
d=O(k+f(k)).
\)
For this fixed topology, the represented graph changes only when the relative order of a
leaf-to-leaf distance and an interval endpoint changes. Since there are $O(k^2)$
leaf-to-leaf distances and $2f(k)$ endpoints, the parameter space is cut by a hyperplane
arrangement of size
\(
H=O(k^2f(k)+f(k)^2).
\)
Because $f(k)\in o(k^2/\log k)$, we in particular have $H=O(k^4)$.

For a fixed reduced leaf-labeled tree topology, let \(d=O(k+f(k))\) denote the number
of real parameters (the edge weights together with the interval endpoints), and let
\(H=O(k^2f(k)+f(k)^2)\) denote the number of relevant comparison hyperplanes. The
parameter space is therefore partitioned by an arrangement of at most \(H\) hyperplanes
in \(\mathbb R^d\). Since the represented graph is constant on each cell of this
arrangement, the number of distinct graphs realizable from this fixed topology is at most
the number of cells.

By the standard bound on the number of cells in a hyperplane arrangement, this number is
at most
\(
\sum_{i=0}^d \binom{H}{i}
\)
(see, e.g., \cite{Zaslavsky1975Facing}). Using
\(
\sum_{i=0}^d \binom{H}{i}
\le (d+1)\binom{H}{d}
\le (d+1)\left(\frac{eH}{d}\right)^d,
\)
we obtain
\(
\sum_{i=0}^d \binom{H}{i}
= e^{O\!\left(d\log(H/d)\right)}.
\)
Now \(d=O(k+f(k))\), and since \(f(k)\in o(k^2/\log k)\), we have
\(
H=O(k^4)
\),
so
\(
\log(H/d)=O(\log k).
\)
Therefore, the number of realizable graphs for this fixed topology is at most
\(
e^{O((k+f(k))\log k)}.
\)

Finally, the number of reduced leaf-labeled tree topologies on \([k]\) is
\(e^{O(k\log k)}\). Multiplying these bounds, we obtain
\(
I_k^{(f)}
\le e^{O(k\log k)}\cdot e^{O((k+f(k))\log k)}
 = e^{O((k+f(k))\log k)}.
\)

On the other hand, the total number of labeled graphs on $[k]$ is
$2^{\binom{k}{2}}=e^{\Theta(k^2)}$. Since $f(k)\in o(k^2/\log k)$, we have
$(k+f(k))\log k=o(k^2)$. Therefore
\[
I_k^{(f)} = e^{o(k^2)} = o\!\left(2^{\binom{k}{2}}\right).
\]
So for all sufficiently large $k$, there exists a graph on $k$ vertices that is not an
\textsc{$f(k)$-interval-\textsc{PCG}}. Since every graph on $k$ vertices is a
\textsc{$(k,t)$-threshold-\textsc{PCG}}, the result follows.
\end{proof}
A particularly relevant consequence is obtained by taking $t=1$ and $f(k)=k$.

\begin{corollary}\label{cor:kinterval-vs-kor}
There exists an integer $k_0$ such that, for every $k\ge k_0$, there exists a graph
that is a \textsc{$k$-\textsc{OR}-\textsc{PCG}} but not a
\textsc{$k$-interval-\textsc{PCG}}. In particular, for every sufficiently large $k$,
\(
k\textsc{-interval-\textsc{PCG}} \subsetneq k\textsc{-OR-\textsc{PCG}}.
\)
\end{corollary}

\begin{proof}
Apply Theorem~\ref{thm:threshold-vs-interval} with $t=1$ and $f(k)=k$. Since
\(
(k,1)\textsc{-threshold-\textsc{PCG}} = k\textsc{-OR-\textsc{PCG}},
\)
and it was shown in \cite{Calamoneri2021OnGOI} that
\(
k\textsc{-interval-\textsc{PCG}} \subset k\textsc{-OR-\textsc{PCG}},
\)
the result follows.
\end{proof}

Thus, we show that, for all sufficiently large $k$, \textsc{$k$-\textsc{OR}-\textsc{PCG}} is strictly
more expressive than \textsc{$k$-interval-\textsc{PCG}}, resolving an earlier open problem \cite{Calamoneri2021OnGOI, survey}.

\section{Set-system Incidence Graphs as Explicit Obstructions}\label{sec:incidence}

The counting arguments of the previous sections show that fixed threshold classes are
asymptotically sparse, but they do not identify concrete obstructions. In the literature
on \textsc{PCG} variants, this distinction is important. The Ramsey-type arguments of
\cite{Calamoneri2021OnGOI} imply that, for every fixed integer $r$, there exist graphs
outside \textsc{$r$-\textsc{OR}-\textsc{PCG}} and outside
\textsc{$r$-\textsc{AND}-\textsc{PCG}}, but those arguments are existential and do not
produce explicit examples. The same paper and the recent survey \cite{survey} highlight
the need for sharper obstruction families, especially for small values of $r$. In
particular, \cite{Calamoneri2021OnGOI} asks for the smallest graph that is not a
\textsc{$2$-\textsc{OR}-\textsc{PCG}} (and similarly for
\textsc{$2$-\textsc{AND}-\textsc{PCG}} and \textsc{$2$-interval-\textsc{PCG}}).

In this section we address this issue by studying bipartite incidence graphs of set
systems. These graphs are particularly well suited to threshold-\textsc{PCG} lower
bounds, because each constituent \textsc{PCG} induces only a controlled number of
neighborhoods on one side. We first prove a shell bound for one constituent
\textsc{PCG}, then combine these bounds under threshold aggregation, and finally
specialize to the family of all $q$-subsets of a $p$-element set. This yields explicit
infinite families outside low-level threshold classes. 
It also applies to the family

\begin{definition}
Let $A$ be a finite set and let $\mathcal S\subseteq 2^A$. The \emph{set-system incidence
graph} of $(A,\mathcal S)$, denoted $G(A,\mathcal S)$, is the bipartite graph with
bipartition
\(
A\cup B,
\)
where
\(
B=\{b_S:S\in\mathcal S\},
\)
and adjacency defined by
\(
ab_S\in E(G(A,\mathcal S))
\ \Longleftrightarrow\ 
a\in S.
\)
Thus, the neighborhoods of the vertices on the $B$-side are precisely the sets in
$\mathcal S$.
\end{definition}

\begin{definition}
Let $A$ be a set of size $p$, and let $0\le q\le p$. The \emph{$(p,q)$-incidence graph},
denoted $I(p,q)$, is the set-system incidence graph
\(
G\!\left(A,\binom{A}{q}\right).
\)
Equivalently, $I(p,q)$ is the bipartite graph with bipartition
\(
A\cup B,
\ 
B=\{b_S:S\subseteq A,\ |S|=q\},
\)
where
\(
N(b_S)=S
\ \text{for each }S\subseteq A,\ |S|=q.
\)
\end{definition}

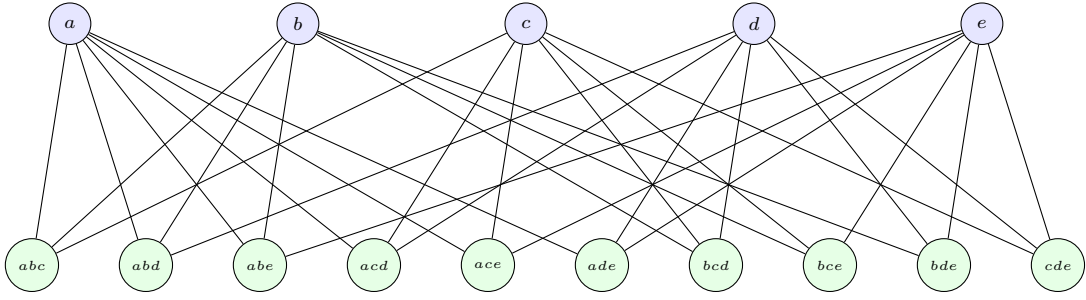
\begin{figure}[t]
\centering
\begin{tikzpicture}[
    x=1cm,y=1cm,
    topv/.style={circle,draw,fill=blue!10,minimum size=5.5mm,inner sep=0pt,font=\scriptsize},
    botv/.style={circle,draw,fill=green!10,minimum size=7mm,inner sep=0pt,font=\tiny},
    edge/.style={thin},
    partlabel/.style={font=\scriptsize\bfseries}
]

\node[topv] (a) at (0.5,3.2) {$a$};
\node[topv] (b) at (3.5,3.2) {$b$};
\node[topv] (c) at (6.5,3.2) {$c$};
\node[topv] (d) at (9.5,3.2) {$d$};
\node[topv] (e) at (12.5,3.2) {$e$};

\node[botv] (abc) at (0,0) {$abc$};
\node[botv] (abd) at (1.5,0) {$abd$};
\node[botv] (abe) at (3,0) {$abe$};
\node[botv] (acd) at (4.5,0) {$acd$};
\node[botv] (ace) at (6,0) {$ace$};
\node[botv] (ade) at (7.5,0) {$ade$};
\node[botv] (bcd) at (9,0) {$bcd$};
\node[botv] (bce) at (10.5,0) {$bce$};
\node[botv] (bde) at (12,0) {$bde$};
\node[botv] (cde) at (13.5,0) {$cde$};

\draw[edge] (a) -- (abc);
\draw[edge] (a) -- (abd);
\draw[edge] (a) -- (abe);
\draw[edge] (a) -- (acd);
\draw[edge] (a) -- (ace);
\draw[edge] (a) -- (ade);

\draw[edge] (b) -- (abc);
\draw[edge] (b) -- (abd);
\draw[edge] (b) -- (abe);
\draw[edge] (b) -- (bcd);
\draw[edge] (b) -- (bce);
\draw[edge] (b) -- (bde);

\draw[edge] (c) -- (abc);
\draw[edge] (c) -- (acd);
\draw[edge] (c) -- (ace);
\draw[edge] (c) -- (bcd);
\draw[edge] (c) -- (bce);
\draw[edge] (c) -- (cde);

\draw[edge] (d) -- (abd);
\draw[edge] (d) -- (acd);
\draw[edge] (d) -- (ade);
\draw[edge] (d) -- (bcd);
\draw[edge] (d) -- (bde);
\draw[edge] (d) -- (cde);

\draw[edge] (e) -- (abe);
\draw[edge] (e) -- (ace);
\draw[edge] (e) -- (ade);
\draw[edge] (e) -- (bce);
\draw[edge] (e) -- (bde);
\draw[edge] (e) -- (cde);

\end{tikzpicture}
\caption{The incidence graph \(I(5,3)\). A vertex \(x\in A = \{a, b, c, d, e\}\) is adjacent to a subset \(S\in \binom{A}{3}\) exactly when \(x\in S\). This is the first graph proven not be a PCG. \cite{Yanhaona2010DiscoveringPCV}}
\label{fig:I53}
\end{figure}

Our first result is a local bound on the number of neighborhoods induced by a single
interval on a fixed tree. For the sake of brevity, we leave the proof for Appendix~\ref{app:ssb}.

\begin{lemma}[Single-shell bound]\label{lem:single-shell-bound}
There exists an absolute constant \(C>0\) such that the following holds.

Let \(T\) be a reduced weighted tree with leaf set \(A\), where \(|A|=m\), and let
\(I=[\alpha,\beta]\) be an interval of real numbers. Define
\[
\mathcal F(T,I):=
\Bigl\{
\{a\in A:d_T(x,a)+\lambda\in I\}: x\in T,\ \lambda\in\mathbb R
\Bigr\},
\]
where \(x\in T\) means that \(x\) is an arbitrary point of the geometric realization of
the tree \(T\). Then
\(
|\mathcal F(T,I)|\le C\,m^3.
\)
\end{lemma}

We now combine these local shell bounds under threshold aggregation.

\begin{theorem}\label{thm:threshold-set-system}
There exists an absolute constant $C>0$ such that the following holds.

Let $A$ be a finite set with $|A|=m$, and let $\mathcal S\subseteq 2^A$. If
$G(A,\mathcal S)$ is a \textsc{$(k,t)$-threshold-\textsc{PCG}}, then
\(
|\mathcal S|\le (C\,m^3)^k.
\)
\end{theorem}

\begin{proof}
Assume that $G(A,\mathcal S)$ is a \textsc{$(k,t)$-threshold-\textsc{PCG}}. Thus there
exist \textsc{PCG}s
\(
G_1,\dots,G_k
\)
on the common vertex set $A\cup B$ such that, for every $a\in A$ and every
$S\in\mathcal S$,
\[
ab_S\in E(G(A,\mathcal S))
\iff
\sum_{i=1}^k \mathbf 1_{\{ab_S\in E(G_i)\}} \ge t.
\]

Fix $i\in[k]$. Since $G_i$ is a \textsc{PCG}, there exist a weighted tree
$\widetilde T_i$, an interval $I_i=[\alpha_i,\beta_i]$, and a leaf-labeling bijection
such that
\(
xy\in E(G_i)
\iff
d_{\widetilde T_i}(x,y)\in I_i
\)
for all distinct leaves corresponding to vertices of $A\cup B$.

Let $T_i$ be the minimal subtree of $\widetilde T_i$ spanning the leaves corresponding
to $A$, and suppress all degree-$2$ vertices. Then $T_i$ is a reduced weighted tree with
leaf set exactly $A$.

For each $S\in\mathcal S$, let $x_{i,S}$ be the unique closest point of $T_i$ to the
leaf corresponding to $b_S$, and let
\(
\lambda_{i,S}:=d_{\widetilde T_i}(b_S,x_{i,S}).
\)
Then for every $a\in A$,
\(
d_{\widetilde T_i}(b_S,a)=\lambda_{i,S}+d_{T_i}(x_{i,S},a),
\)
and hence
\(
ab_S\in E(G_i)
\iff
d_{T_i}(x_{i,S},a)+\lambda_{i,S}\in I_i.
\)

Define
\(
\mathcal F_i:=
\Bigl\{
\{a\in A:d_{T_i}(x,a)+\lambda\in I_i\}: x\in T_i,\ \lambda\in\mathbb R
\Bigr\}.
\)
By Lemma~\ref{lem:single-shell-bound},
\(
|\mathcal F_i|\le C\,m^3
\ \text{for each }i\in[k].
\)

Now fix $S\in\mathcal S$, and let
\(
X_{i,S}:=\{a\in A:ab_S\in E(G_i)\}.
\)
Then $X_{i,S}\in\mathcal F_i$ for each $i\in[k]$, and by the threshold rule,
\(
S=
\Bigl\{
a\in A:
\sum_{i=1}^k \mathbf 1_{\{a\in X_{i,S}\}} \ge t
\Bigr\}.
\)
Therefore, each set $S\in\mathcal S$ is determined by an $k$-tuple
\(
(X_{1,S},\dots,X_{k,S})\in \mathcal F_1\times\cdots\times\mathcal F_k.
\)
Hence,
\[
|\mathcal S|
\le
|\mathcal F_1|\cdots |\mathcal F_k|
\le
(C\,m^3)^k.
\] \end{proof}

Applying this to the complete $q$-uniform set system yields the main explicit lower
bound.

\begin{corollary}\label{cor:threshold-incidence}
There exists an absolute constant $C>0$ such that, if $I(p,q)$ is a
\textsc{$(k,t)$-threshold-\textsc{PCG}}, then
\(
\binom{p}{q}\le (C\,p^3)^k.
\)
Consequently,
\(
k\ge \frac{\log \binom{p}{q}}{\log C + 3\log p}.
\)
\end{corollary}

\begin{proof}
Apply Theorem~\ref{thm:threshold-set-system} with
\(
m=p
\ \text{and}\ 
\mathcal S=\binom{A}{q}.
\)
Then
\(
|\mathcal S|=\binom{p}{q},
\)
which gives the first inequality; the second follows by taking logarithms.
\end{proof}

The case $k=2$ is already enough to produce explicit low-level obstructions.

\begin{corollary}\label{cor:explicit-2or}
For every sufficiently large $p$, the graph $I(p,7)$ is neither a
\textsc{$2$-\textsc{OR}-\textsc{PCG}} nor a \textsc{$2$-\textsc{AND}-\textsc{PCG}}.
More generally, for every sufficiently large $p$, the graph $I(p,7)$ is not a
\textsc{$(2,t)$-threshold-\textsc{PCG}} for either $t=1$ or $t=2$.
\end{corollary}

\begin{proof}
If $I(p,7)$ were a \textsc{$(2,t)$-threshold-\textsc{PCG}}, then by
Corollary~\ref{cor:threshold-incidence},
\(
2\ge \frac{\log \binom{p}{7}}{\log C + 3\log p}.
\)
But
\(
\log \binom{p}{7}=7\log p + O(1),
\)
so the right-hand side tends to $7/3>2$, a contradiction for sufficiently large $p$.
The cases $t=1$ and $t=2$ are precisely \textsc{$2$-\textsc{OR}-\textsc{PCG}} and
\textsc{$2$-\textsc{AND}-\textsc{PCG}}, respectively.
\end{proof}

Corollary~\ref{cor:explicit-2or} should be compared with Problem~5 of
\cite{Calamoneri2021OnGOI}, which asks for the smallest graph that is not a
\textsc{$2$-\textsc{OR}-\textsc{PCG}} (and similarly for
\textsc{$2$-\textsc{AND}-\textsc{PCG}} and \textsc{$2$-interval-\textsc{PCG}}). Our
result does not optimize the order of such a graph, but it does provide an explicit
infinite family of non-\textsc{$2$-\textsc{OR}-\textsc{PCG}} and
non-\textsc{$2$-\textsc{AND}-\textsc{PCG}} examples, in contrast with the purely
existential Ramsey-type argument of \cite{Calamoneri2021OnGOI}.

We next turn to the family
\(
H_y:=I(4y-3,\,2y-1),
\qquad y\ge 1.
\)

\begin{corollary}\label{cor:Hk-threshold}
For every integer \(y\ge 1\), any \textsc{$(k,t)$-threshold-\textsc{PCG}}
representation of \(H_y\) satisfies
\(
k=\Omega\!\left(\frac{y}{\log y}\right).
\)
\end{corollary}

\begin{proof}
Suppose that \(H_y\) is a \textsc{$(k,t)$-threshold-\textsc{PCG}}. Applying
Corollary~\ref{cor:threshold-incidence} with
\(
p=4y-3
\ \text{and}\ 
q=2y-1,
\)
we obtain
\(
k\ge \frac{\log \binom{p}{q}}{\log C + 3\log p}.
\)
Since
\(
q=\frac12 p+O(1),
\)
we have
\(
\log \binom{p}{q}=\Theta(p)=\Theta(y),
\)
whereas
\(
\log C + 3\log p=\Theta(\log p)=\Theta(\log y).
\)
Therefore,
\(
k\ge \frac{\log \binom{p}{q}}{\log C + 3\log p}
  = \Omega\!\left(\frac{y}{\log y}\right),
\)
as claimed.
\end{proof}

For comparison, the same shell-counting method yields a much stronger lower bound in the
multi-interval model.

\begin{proposition}\label{prop:interval-set-system}
There exists an absolute constant \(C>0\) such that the following holds.

Let \(A\) be a finite set with \(|A|=m\), and let \(\mathcal S\subseteq 2^A\). If
\(G(A,\mathcal S)\) is an \textsc{$\ell$-interval-\textsc{PCG}}, then
\(
|\mathcal S|\le C\,m^3\ell^2.
\)
\end{proposition}

\begin{proof}
Deferred to Appendix~\ref{app:ssb}.
\end{proof}


\begin{corollary}\label{cor:Hk-interval}
If
\(
H_y=I(4y-3,\,2y-1)
\)
is an \textsc{$\ell$-interval-\textsc{PCG}}, then
\(
\ell=\Omega\!\left(\frac{4^y}{y^{7/4}}\right).
\)
In particular, \(H_y\notin y\)-interval-\textsc{PCG} for all sufficiently large \(y\).
\end{corollary}

\begin{proof}
Apply Proposition~\ref{prop:interval-set-system} with
\(
m=4y-3
\ \text{and}\ 
|\mathcal S|=\binom{4y-3}{2y-1}.
\)
We obtain
\(
\binom{4y-3}{2y-1}\le C\,(4y-3)^3\ell^2,
\ \text{and hence}\ 
\ell\ge
\sqrt{\frac{\binom{4y-3}{2y-1}}{C(4y-3)^3}}.
\)
Using the central binomial estimate
\(
\binom{4y-3}{2y-1}=\Theta\!\left(\frac{16^y}{\sqrt{y}}\right)
\) (see, e.g.\cite{Graham1994Concrete}), 
it follows that
\(
\ell=\Omega\!\left(\frac{4^y}{y^{7/4}}\right).
\)
Since \(\frac{4^y}{y^{7/4}}\) eventually exceeds \(y\), the final claim follows.
\end{proof}

The family \(H_y\) already appears in \cite{Calamoneri2021OnGOI} as a hard family for
multi-interval representations: that paper proves that
\(
H_y\notin (y-1)\text{-interval-}\textsc{PCG}
\)
and leaves open the problem of determining the minimum integer \(\ell\ge y\) such that
\(
H_y\in \ell\text{-interval-}\textsc{PCG}.
\)
Corollary~\ref{cor:Hk-interval} gives a much stronger lower bound on the interval side,
while Corollary~\ref{cor:Hk-threshold} shows that the same family remains difficult even
under the more flexible threshold model. Thus \(H_y\) serves as an explicit obstruction
family simultaneously for both multi-interval and threshold-based generalizations of
\textsc{PCG}.

\begin{remark}
The threshold parameter $t$ does not affect the bound in
Theorem~\ref{thm:threshold-set-system}. Once the $k$ constituent \textsc{PCG}s are fixed,
each neighborhood on the $B$-side is determined by an $k$-tuple of shell sets, one from
each constituent representation. The threshold rule only specifies how these $k$ sets
are combined pointwise, and therefore cannot increase the total number of possible
neighborhoods beyond the number of available $k$-tuples.
\end{remark}

\section{The \textsc{$k$-\textsc{AND}-\textsc{PCG}} Hierarchy}\label{sec:and}

We now specialize to the conjunction setting. We show that the hierarchy of
\textsc{$k$-\textsc{AND}-\textsc{PCG}} classes is strict, and that none of these classes
is closed under complement. These two questions were posed as open problems in \cite{Calamoneri2021OnGOI, survey}.

\begin{theorem}\label{thm:and-hierarchy}
For every integer $k \ge 1$,
\(
k\textsc{-AND-\textsc{PCG}} \subsetneq (k+1)\textsc{-AND-\textsc{PCG}}.
\)
\end{theorem}

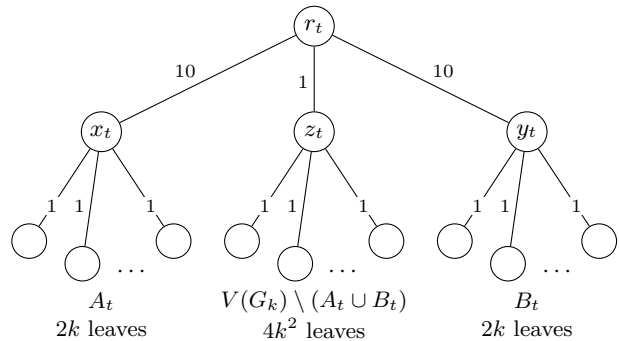
\begin{figure}[b]
\centering
\begin{tikzpicture}[
    every node/.style={font=\small},
    v/.style={circle,draw,inner sep=1.1pt,minimum size=15pt},
    leaf/.style={circle,draw,inner sep=0.9pt,minimum size=13pt},
    lbl/.style={font=\footnotesize},
    edge label/.style={font=\scriptsize,inner sep=1pt,fill=white}
]

\node[v] (r) at (0,2.2) {$r_t$};
\node[v] (x) at (-2.8,0.8) {$x_t$};
\node[v] (z) at ( 0.0,0.8) {$z_t$};
\node[v] (y) at ( 2.8,0.8) {$y_t$};

\draw (r) -- node[pos=.55,above left=1pt and 0pt,edge label] {$10$} (x);
\draw (r) -- node[pos=.55,left=1pt,edge label] {$1$} (z);
\draw (r) -- node[pos=.55,above right=1pt and 0pt,edge label] {$10$} (y);

\node[leaf] (a1) at (-3.75,-0.65) {};
\node[leaf] (a2) at (-3.05,-0.95) {};
\node[leaf] (a3) at (-1.85,-0.65) {};

\draw (x) -- node[pos=.58,below left=1pt and 0pt,edge label] {$1$} (a1);
\draw (x) -- node[pos=.57,left=1pt,edge label] {$1$} (a2);
\draw (x) -- node[pos=.58,below right=1pt and 0pt,edge label] {$1$} (a3);

\node[fill=white,inner sep=1pt] at (-2.38,-1.05) {$\cdots$};

\node[lbl] at (-2.8,-1.45) {$A_t$};
\node[lbl] at (-2.8,-1.80) {$2k$ leaves};

\node[leaf] (c1) at (-0.95,-0.65) {};
\node[leaf] (c2) at (-0.25,-0.95) {};
\node[leaf] (c3) at ( 0.95,-0.65) {};

\draw (z) -- node[pos=.58,below left=1pt and 0pt,edge label] {$1$} (c1);
\draw (z) -- node[pos=.57,left=1pt,edge label] {$1$} (c2);
\draw (z) -- node[pos=.58,below right=1pt and 0pt,edge label] {$1$} (c3);

\node[fill=white,inner sep=1pt] at (0.40,-1.05) {$\cdots$};

\node[lbl] at (0.0,-1.45) {$V(G_k)\setminus(A_t\cup B_t)$};
\node[lbl] at (0.0,-1.80) {$4k^2$ leaves};

\node[leaf] (b1) at (1.85,-0.65) {};
\node[leaf] (b2) at (2.55,-0.95) {};
\node[leaf] (b3) at (3.75,-0.65) {};

\draw (y) -- node[pos=.58,below left=1pt and 0pt,edge label] {$1$} (b1);
\draw (y) -- node[pos=.57,left=1pt,edge label] {$1$} (b2);
\draw (y) -- node[pos=.58,below right=1pt and 0pt,edge label] {$1$} (b3);

\node[fill=white,inner sep=1pt] at (3.22,-1.05) {$\cdots$};

\node[lbl] at (2.8,-1.45) {$B_t$};
\node[lbl] at (2.8,-1.80) {$2k$ leaves};

\end{tikzpicture}
\caption{The weighted tree \(T_t\) realizing \(Q_t=K_n-E(K^{(t)}_{2k,2k})\). Leaves attached to \(x_t\) correspond to \(A_t\), leaves attached to \(y_t\) correspond to \(B_t\), and all remaining vertices are attached to \(z_t\).}
\label{fig:Qt-tree}
\end{figure}

\begin{proof}
The inclusion
\(
k\textsc{-AND-\textsc{PCG}} \subseteq (k+1)\textsc{-AND-\textsc{PCG}}
\)
is immediate: if
\(
G=\bigcap_{i=1}^k G_i
\)
with each $G_i$ a \textsc{PCG}, then
\(
G=\left(\bigcap_{i=1}^k G_i\right)\cap K_{n},
\)
and the complete graph $K_{n}$ is itself a \textsc{PCG}. It remains to show that the inclusion is proper.

For fixed $k\ge 1$, let
\(
F_k:=\bigsqcup_{t=1}^{k+1} K^{(t)}_{2k,2k},
\)
the disjoint union of $k+1$ copies of $K_{2k,2k}$, and let
\(
G_k:=\overline{F_k}.
\)
We prove that
\(
G_k\in (k+1)\textsc{-AND-\textsc{PCG}}\) and \(G_k\notin k\textsc{-AND-\textsc{PCG}}.
\)

\medskip
\noindent\textbf{Step 1: \(G_k\in (k+1)\)-AND-PCG.}

For each $t\in\{1,\dots,k+1\}$, let $A_t,B_t$ be the two partite sets of the $t$-th copy
\(K^{(t)}_{2k,2k}\), so that \(|A_t|=|B_t|=2k\). Define
\(
Q_t:=K_n-E\!\left(K^{(t)}_{2k,2k}\right),
\)
where \(n=|V(G_k)|=4k(k+1)\). Then
\(
G_k=\bigcap_{t=1}^{k+1}Q_t.
\)
Thus it suffices to show that each \(Q_t\) is a \textsc{PCG}.

Fix \(t\). Construct a weighted tree \(T_t\) as follows:
\begin{itemize}
    \item start with a central vertex \(r_t\);
    \item attach three internal vertices \(x_t,y_t,z_t\) to \(r_t\) with edge weights
    \(
    w(r_tx_t)=10,\quad   w(r_ty_t)=10,\quad w(r_tz_t)=1;
    \)
    \item attach each vertex of \(A_t\) as a leaf adjacent to \(x_t\) by an edge of weight \(1\);
    \item attach each vertex of \(B_t\) as a leaf adjacent to \(y_t\) by an edge of weight \(1\);
    \item attach every remaining vertex as a leaf adjacent to \(z_t\) by an edge of weight \(1\).
\end{itemize}

The leaf-to-leaf distances in \(T_t\) are:
\[
d_{T_t}(u,v)=
\begin{cases}
2, & \text{if } \{u,v\}\subseteq A_t,\ \text{or } \{u,v\}\subseteq B_t,\ \text{or } \{u,v\}\subseteq V(G_k)\setminus (A_t\cup B_t),\\[4pt]
13, & \text{if } u\in A_t,\ v\in V(G_k)\setminus (A_t\cup B_t),\ \text{or } v\in A_t,\ u\in V(G_k)\setminus (A_t\cup B_t),\\
& \text{or } u\in B_t,\ v\in V(G_k)\setminus (A_t\cup B_t),\ \text{or } v\in B_t,\ u\in V(G_k)\setminus (A_t\cup B_t),\\[4pt]
22, & \text{if } u\in A_t,\ v\in B_t,\ \text{or } u\in B_t,\ v\in A_t.
\end{cases}
\]
\[
\text{Therefore, }
Q_t=\operatorname{PCG}(T_t,[2,13]).
\text{ and hence, }
G_k=\bigcap_{t=1}^{k+1}Q_t\in (k+1)\textsc{-AND-\textsc{PCG}}.\hfill\triangleleft
\]

\medskip
\noindent\textbf{Step 2: \(G_k\notin k\)-AND-PCG.}

Assume for contradiction that
\(
G_k=\bigcap_{i=1}^k P_i
\)
for some \textsc{PCGs} \(P_1,\dots,P_k\).

Since
\(
\overline{G_k}=F_k,
\)
every edge of \(F_k\) is absent from at least one \(P_i\). For each edge \(e\in E(F_k)\),
choose one index
\(
c(e)\in\{1,\dots,k\}
\)
such that \(e\notin P_{c(e)}\), and color \(e\) with color \(c(e)\).

Fix one component \(K_{2k,2k}\) of \(F_k\). It has \(4k\) vertices and \(4k^2\) edges.
If every color class inside this component were acyclic, then each color class would be a
forest on \(4k\) vertices, and hence would contain at most \(4k-1\) edges. Therefore, the
total number of edges in the component would be at most
\(
k(4k-1)=4k^2-k<4k^2,
\)
a contradiction. Hence, in each component \(K_{2k,2k}\), there exists some color \(i\)
such that the edges of color \(i\) contain a cycle.

Since there are \(k+1\) components and only \(k\) colors, by the pigeonhole principle
there exists a color \(i\) that contains a cycle in at least two distinct components of
\(F_k\).

Now every edge of color \(i\) is absent from \(P_i\), so it belongs to \(\overline{P_i}\).
Also,
\(
\overline{P_i}\subseteq \overline{G_k}=F_k,
\)
because if an edge were absent from \(P_i\) but belonged to
\(
G_k=\bigcap_{j=1}^k P_j,
\)
then in particular it would belong to \(P_i\), impossible.

Therefore, \(\overline{P_i}\) contains a cycle in each of two different connected
components of \(F_k\). In each of those two components, choose a shortest cycle of
\(\overline{P_i}\). Each such cycle is chordless. Since the two cycles lie in different
components of \(F_k\), they are vertex-disjoint and there are no edges between them.

By the necessary condition for \textsc{PCG}s proved by Hossain et al., if the complement
of a graph contains two vertex-disjoint chordless cycles such that no vertices on one cycle is adjacent to the vertices on the other cycle, then that graph is not a
\textsc{PCG} \cite{Hossain2016ANCAI,Rahman2020ASOA}. Applying this to \(P_i\), we obtain
that \(P_i\) is not a \textsc{PCG}, a contradiction. Hence \(G_k\notin k\textsc{-AND-\textsc{PCG}}\).\hfill$\triangleleft$

Combining Steps 1 and 2, we obtain
\(
k\textsc{-AND-\textsc{PCG}} \subsetneq (k+1)\textsc{-AND-\textsc{PCG}}.
\)
\end{proof}

\begin{theorem}\label{thm:and-not-complement}
For every integer $k\ge 1$, the class \(k\)-\textsc{AND}-\textsc{PCG} is not closed
under complement.
\end{theorem}

\begin{proof}
Let
\(
F_k:=\bigsqcup_{t=1}^{k+1} K^{(t)}_{2k,2k}
\text{ and }
G_k:=\overline{F_k}.
\)
Each graph \(K_{2k,2k}\) is a \textsc{PCG} \cite{Yanhaona2010DiscoveringPCV}, and finite disjoint unions of
\textsc{PCG}s are again \textsc{PCG}s \cite{XIAO2020105875}. Hence
\(
F_k\in \textsc{PCG}\subseteq k\textsc{-AND-\textsc{PCG}}.
\)
On the other hand, by Theorem~\ref{thm:and-hierarchy},
\(
G_k=\overline{F_k}\notin k\textsc{-AND-\textsc{PCG}}.
\)
Therefore, \(k\)-\textsc{AND}-\textsc{PCG} is not closed under complement.
\end{proof}

\section{Conclusions and Open Problems}\label{sec:conclusion}

We introduced \textsc{$(k,t)$-threshold-\textsc{PCG}}s as a common framework for several
existing generalizations of \textsc{PCGs}. This viewpoint simultaneously recovers
ordinary \textsc{PCGs}, \textsc{$k$-\textsc{OR}-\textsc{PCG}}s, and
\textsc{$k$-\textsc{AND}-\textsc{PCG}}s, while remaining close to the original
tree-distance model.

Our results exhibit a clear dichotomy. When the number of constituent predicates grows
with the graph size, the model becomes universal: every graph on \(n\) vertices is a
\textsc{$(n,t)$-threshold-\textsc{PCG}} for every \(t\in[n]\). In contrast, for
fixed \(k\), even after allowing all thresholds \(t\in[k]\), the resulting classes
have asymptotic density zero. We also showed that threshold combinations can be much
more expressive than multi-interval representations: if \(f(k)\in o(k^2/\log k)\), then
for all sufficiently large \(k\) there exists a graph that is a
\textsc{$(k,t)$-threshold-\textsc{PCG}} but not an
\textsc{$f(k)$-interval-\textsc{PCG}}. In particular, for all sufficiently large \(k\),
\(
k\textsc{-interval-\textsc{PCG}} \subsetneq k\textsc{-OR-\textsc{PCG}}.
\)

Beyond counting arguments, we developed explicit obstruction families via set-system
incidence graphs. These yield concrete lower bounds for threshold representations and, in
particular, explicit infinite families outside low-level threshold classes. The same
approach also gives strong lower bounds for the family \(H_m\) in the multi-interval
setting. Finally, for conjunctions, we proved that
\(
k\textsc{-AND-\textsc{PCG}} \subsetneq (k+1)\textsc{-AND-\textsc{PCG}}
\)
for every \(k\ge 1\), and that \(k\)-\textsc{AND}-\textsc{PCG} is not closed under
complement.

The main remaining questions are quantitative. Our results leave a substantial gap
between lower and upper bounds for explicit obstruction families, and they also leave
open the small-parameter behavior of several threshold and multi-interval classes.

\paragraph*{Open problems.}
\begin{enumerate}
    \item Determine the correct order of growth of the minimum threshold width of
    \(
    H_y=I(4y-3,\,2y-1).
    \)
    Our results give a lower bound of \(\Omega(y/\log y)\), while the general universality theorem gives only much weaker upper bounds.

    \item Determine the correct order of growth of the minimum interval number needed to
    represent \(H_y\). We proved the lower bound
    \(
    \Omega\!\left(\frac{4^y}{y^{7/4}}\right),
    \)
    but the exact asymptotics remain unknown.

    \item Sharpen the incidence-graph bounds. For general \(I(p,q)\), determine the true
    minimum number of constituent \textsc{PCG}s required in a
    \textsc{$(k,t)$-threshold-\textsc{PCG}} representation, and compare it to the true
    minimum interval number in the multi-interval model.

    \item Determine whether the separation
    \(
    k\textsc{-interval-\textsc{PCG}} \subsetneq k\textsc{-OR-\textsc{PCG}}
    \)
    already holds for every fixed \(k\ge 2\), or only asymptotically for sufficiently
    large \(k\).
    \item Understand the relationship between the PCG parameters such as minimum threshold width or minimum internal number and the standard graph parameters, such as treewidth, pathwidth, rankwidth, twinwidth, branchwidth, etc. 
\end{enumerate}

\bibliography{lipics-v2021-sample-article}

\appendix

\section{Non-closure of Multi-Interval PCGs under Complement}
\label{app:kint}

It has been posed as an open problem in \cite{Calamoneri2021OnGOI,survey} to determine whether $k$\textsc{-interval-\textsc{PCG}} is closed under complement. We resolved this question in the negative, utilizing the hierarchy of generalized leaf powers introduced by Dupr\'e la Tour,
Lafond, and Ndiaye~\cite{la2026recognizing}.

\begin{definition}[Leaf power \cite{Calamoneri2011OnRTB}]
A graph is a \emph{leaf power} if it is a \textsc{PCG} representable as
\[
\operatorname{PCG}(T,[0,d_{\max}])
\]
for some edge-weighted tree \(T\) and some \(d_{\max}\ge 0\).
\end{definition}

\begin{definition}[Generalized leaf powers \cite{la2026recognizing}]
Let \(q \ge 1\) be an integer. A graph \(G=(V,E)\) is a \(\mathrm{GLP}(q)\) if there exist
a positively weighted tree \(T\) with leaf set \(V\), and real thresholds
\[
\theta_1 < \theta_2 < \cdots < \theta_q,
\]
such that, for every two distinct vertices \(u,v \in V\),
\[
uv \in E
\quad\Longleftrightarrow\quad
d_T(u,v)\ \text{is at most an odd number of the thresholds } \theta_1,\dots,\theta_q.
\]
\end{definition}

To write our proof succintly, we define a new graph operation. For a graph \(G\), let \(G_1\) and \(G_2\) denote two disjoint copies of \(G\), and we write
\(
G' := \overline{G_1 \cup G_2}.
\)

We now state several well-known facts to facilitate our proof.

\begin{proposition}\label{fact:glp-even}
\cite{la2026recognizing}
For every integer \(i \ge 1\),
\[
\mathrm{GLP}(1)=\textsc{Leaf Power}
\qquad\text{and}\qquad
\mathrm{GLP}(2i)= i\text{-interval-}\textsc{PCG}.
\]
\end{proposition}

\begin{proposition}\label{fact:glp-up}
\cite{la2026recognizing}
For every integer \(q \ge 1\), if \(G \in \mathrm{GLP}(q)\), then
\[
G'=\overline{G_1 \cup G_2}\in \mathrm{GLP}(q+1).
\]
\end{proposition}

\begin{proposition}\label{fact:glp-down}
\cite{la2026recognizing}
For every integer \(q \ge 1\), if \(G \notin \mathrm{GLP}(q)\), then
\[
G'=\overline{G_1 \cup G_2}\notin \mathrm{GLP}(q+1).
\]
\end{proposition}

\begin{proposition}\label{fact:c4-not-glp1}
\cite{la2026recognizing}
\(
C_4 \notin \mathrm{GLP}(1).
\)
\end{proposition}

We can now prove our theorem. 

\begin{theorem}
For every integer \(k \ge 1\), the class of \(k\)-interval-\textsc{PCG}s is not closed
under complement.
\end{theorem}

\begin{proof}
Fix \(k \ge 1\). For \(r \ge 1\), define graphs \(F_r\) recursively by
\[
F_1 := C_4,
\qquad
F_{r+1} := \overline{F_r^{(1)} \cup F_r^{(2)}},
\]
where \(F_r^{(1)}\) and \(F_r^{(2)}\) are two disjoint copies of \(F_r\). By Fact~\ref{fact:c4-not-glp1}, we have \(F_1 \notin \mathrm{GLP}(1)\). Hence,
by repeated application of Fact~\ref{fact:glp-down},
\[
F_r \notin \mathrm{GLP}(r)
\qquad\text{for every } r \ge 1.
\]
In particular,
\[
F_{2k} \notin \mathrm{GLP}(2k).
\]
Using Fact~\ref{fact:glp-even}, this becomes
\[
F_{2k} \notin k\text{-interval-}\textsc{PCG}.
\]

Now suppose, for contradiction, that the class of \(k\)-interval-\textsc{PCG}s is
closed under complement.

Since \(C_4\) is a \textsc{PCG}~\cite{Calamoneri2012AllGWN}, it is also a
$k$\textsc{-interval-\textsc{PCG}}. Thus \(F_1\) belongs to the class.
We claim that then \(F_r\) is a $k$\textsc{-interval-\textsc{PCG}} for every \(r \ge 1\).

Indeed, assume that \(F_r\) is a $k$\textsc{-interval-\textsc{PCG}}.
Then \(F_r^{(1)}\) and \(F_r^{(2)}\) are also \(k\)-interval-\textsc{PCG}s.
Moreover, their disjoint union is again a $k$\textsc{-interval-\textsc{PCG}}: take two
disjoint copies of a witness tree for \(F_r\) and connect them by an edge of
sufficiently large weight, so that every leaf-to-leaf distance between the two copies
lies outside the union of the \(k\) admissible intervals, while all distances inside
each copy remain unchanged. Therefore,
\[
F_r^{(1)} \cup F_r^{(2)}
\]
is a $k$\textsc{-interval-\textsc{PCG}}. By the assumed closure under complement,
\[
F_{r+1}=\overline{F_r^{(1)} \cup F_r^{(2)}}
\]
is also a $k$\textsc{-interval-\textsc{PCG}}. This proves the claim by induction.

Applying the claim with \(r=2k\), we obtain that \(F_{2k}\) is a
$k$\textsc{-interval-\textsc{PCG}}, resulting in a contradiction.

Therefore, the class of \(k\)-interval-\textsc{PCG}s is not closed under complement.
\end{proof}
\section{Shell Bound Lemma
}
\label{app:ssb}

In this section, we introduce some useful notions that are useful for proving several results on the incidence graph classes.

\begin{lemma}[Single-shell bound]\label{lem:single-shell-bound_app}
There exists an absolute constant \(C>0\) such that the following holds.

Let \(T\) be a reduced weighted tree with leaf set \(A\), where \(|A|=m\), and let
\(I=[\alpha,\beta]\) be an interval of real numbers. Define
\[
\mathcal F(T,I):=
\Bigl\{
\{a\in A:d_T(x,a)+\lambda\in I\}: x\in T,\ \lambda\in\mathbb R
\Bigr\},
\]
where \(x\in T\) means that \(x\) is an arbitrary point of the geometric realization of
the tree \(T\). Then
\(
|\mathcal F(T,I)|\le C\,m^3.
\)
\end{lemma}

\begin{proof}
Fix an interval \(I=[\alpha,\beta]\). For a point \(x\) on the tree and a real number \(\lambda\), define
\[
X(x,\lambda):=\{a\in A: d_T(x,a)+\lambda\in I\}.
\]
We want to bound how many different subsets \(X(x,\lambda)\) can appear. We first keep \(x\) on one fixed edge \(e\) of the tree, and count how many subsets can be realized there.

Let \(e\) be an edge of length \(\ell_e\), and parametrize its points by
\(x\in[0,\ell_e]\). For each leaf \(a\in A\), the restriction of \(d_T(x,a)\) to \(e\)
is affine with slope in \(\{+1,-1\}\): as \(x\) moves along \(e\), the distance to \(a\)
either increases at unit rate or decreases at unit rate, depending on which side of \(e\)
the leaf \(a\) lies. Hence
\[
d_T(x,a)=\sigma_a x+c_a
\]
for suitable constants \(\sigma_a\in\{+1,-1\}\) and \(c_a\in\mathbb R\).

Now fix a leaf \(a\). The membership of \(a\) in \(X(x,\lambda)\) is determined by whether
\(
\alpha \le d_T(x,a)+\lambda \le \beta.
\)
As \((x,\lambda)\) varies, this truth value can change only when
\(d_T(x,a)+\lambda\) hits one of the two endpoints \(\alpha\) or \(\beta\). Thus, the only
possible change points are on the two lines
\[
\lambda=\alpha-(\sigma_a x+c_a)
\qquad\text{and}\qquad
\lambda=\beta-(\sigma_a x+c_a).
\]

So inside the strip
\(
0\le x\le \ell_e,
\)
the behavior of all leaves is controlled by an arrangement of at most \(2m\) such lines,
together with the two boundary lines \(x=0\) and \(x=\ell_e\).

This arrangement cuts the strip into \(O(m^2)\) regions. Inside any one region, no boundary
line is crossed, so for every leaf \(a\), the statement
\(
d_T(x,a)+\lambda\in I
\)
has the same truth value throughout the whole region. 

Therefore, the subset \(X(x,\lambda)\)
is constant on each region. Hence, while \(x\) stays on the fixed edge \(e\), at most \(O(m^2)\) different subsets can
be realized.

Finally, since \(T\) is reduced and has \(m\) leaves, it has at most \(2m-3\) edges.
Summing the \(O(m^2)\) possibilities over all edges gives
\(
|\mathcal F(T,I)|=O(m^3)
\), and the claim follows.
\end{proof}

\begin{proposition}\label{prop:interval-set-system_app}
There exists an absolute constant \(C>0\) such that the following holds.

Let \(A\) be a finite set with \(|A|=m\), and let \(\mathcal S\subseteq 2^A\). If
\(G(A,\mathcal S)\) is an \textsc{$\ell$-interval-\textsc{PCG}}, then
\(
|\mathcal S|\le C\,m^3\ell^2.
\)
\end{proposition}

\begin{proof}
The proof is the same shell-counting argument as in
Lemma~\ref{lem:single-shell-bound_app}, but with a union of \(\ell\) disjoint intervals
instead of a single interval. On a fixed edge of the reduced tree, each leaf contributes
\(2\ell\) endpoint lines in the \((x,\lambda)\)-plane, so the arrangement has
\(O(m\ell)\) lines and therefore \(O(m^2\ell^2)\) faces. Summing over the \(O(m)\)
edges of the reduced tree yields \(O(m^3\ell^2)\) possible neighborhoods on the
\(B\)-side.
\end{proof}

\end{document}